\documentclass[11pt,a4paper,reqno]{amsart}
\usepackage[left=30mm,top=40mm,right=30mm,bottom=35mm]{geometry}
\usepackage{times}
\usepackage[T1]{fontenc}
\usepackage[utf8]{inputenc}
\usepackage[english]{babel}
\usepackage[normalem]{ulem}

\usepackage{amssymb,amsmath,amsfonts,amsthm,amscd}
\usepackage{indentfirst,graphicx}
\usepackage{graphics,graphicx,epsf}

\usepackage[all]{xy}
\usepackage{tikz}
\usepackage{color,xcolor}

\usepackage{hyperref}
\usepackage{comment}
\usepackage{cite}
\usepackage[nameinlink,capitalize,noabbrev]{cleveref}

\usepackage{mathabx} 
\usepackage{float} 
\usepackage{faktor}

\usepackage{enumerate}
\usepackage{multicol}

\usepackage{thmtools}
\usepackage{thm-restate}

\usepackage[font={scriptsize},captionskip=5pt]{subfig}
\usetikzlibrary{matrix}

\usetikzlibrary{patterns}
\usetikzlibrary{patterns.meta}
\tikzdeclarepattern{
	name=hatch,
	parameters={\hatchsize,\hatchangle,\hatchlinewidth},
	bounding box={(-.1pt,-.1pt) and (\hatchsize+.1pt,\hatchsize+.1pt)},
	tile size={(\hatchsize,\hatchsize)},
	tile transformation={rotate=\hatchangle},
	defaults={
		hatch size/.store in=\hatchsize,hatch size=5pt,
		hatch angle/.store in=\hatchangle,hatch angle=0,
		hatch linewidth/.store in=\hatchlinewidth,hatch linewidth=.4pt,
	},
	code={
		\draw[line width=\hatchlinewidth] (0,0) -- (\hatchsize,\hatchsize);
	}
}

\newtheorem{theorem}{Theorem}[section]

\newtheorem{corollary}[theorem]{Corollary}

\newtheorem{definition}[theorem]{Definition}

\newtheorem{example}[theorem]{Example}
\newtheorem{remark}[theorem]{Remark}

\newcommand{\C}{\ensuremath{\mathbb{C}}}
\newcommand{\R}{\ensuremath{\mathbb{R}}}

\newcommand{\Z}{\ensuremath{\mathbb{Z}}}

\newcommand{\Poli}{\ensuremath{\Gamma_{+}}}
\newcommand{\On}{\ensuremath{\mathcal{O}_{n}}}
\newcommand{\OX}[1]{\ensuremath{\mathcal{O}_{X(#1)}}}
\newcommand{\coned}[1]{\ensuremath{c\widecheck{on}e(#1)}}

\begin{document}
	
	\title[Bi-Lipschitz Invariants in Singularity Theory: \L ojasiewicz Exponent and Euler Obstruction
	]{Bi-Lipschitz Invariants in Singularity Theory: \L ojasiewicz Exponent and Euler Obstruction
	}
	\author[A. S. Araújo]{Amanda S. Araújo}
	\address{Departamento de Matemática, Universidade Federal de São Carlos - UFSCar, Brazil}
	\email{amandaaraujo@estudante.ufscar.br}
	
	\author[T. M. Dalbelo]{Thaís M. Dalbelo}
	\address{Departamento de Matemática, Universidade Federal de São Carlos - UFSCar,Brazil}
	\email{thaisdalbelo@ufscar.br}
	
	\author[T. da Silva]{Thiago da Silva}
	\address{Departamento de Matemática, Universidade Federal do Espírito Santo - UFES, Brazil}
	\email{thiago.silva@ufes.br}

	\thanks{
		\mbox{\hspace{11pt}}{\it Key words}:  Bi-Lipschitz equivalence, Łojasiewicz exponent, Euler obstruction.\\
		\mbox{\hspace{11pt}} Amanda S. Araújo is supported by CAPES grant number 88887.827300/2023-00.
		Thaís M. Dalbelo is supported by FAPESP Grants 2019/21181-0 and 2024/22060-0, and by CNPq Grant 403959/2023-3. Thiago da Silva is funded by CAPES grant number 88887.909401/2023-00, and by FAPES (Universal: Edital FAPES 13/2025).}

	\begin{abstract}
		In this work, we investigate the bi-Lipschitz invariance of two fundamental local invariants in singularity theory: the Łojasiewicz exponent and the local Euler obstruction. We draw inspiration from Bivià-Ausina and Fukui, whose framework we extend to ideals in rings of analytic functions defined on affine toric varieties. We establish conditions under which these invariants remain unchanged under bi-Lipschitz equivalence. We also provide an answer, to a particular case, to the open question of whether the local Euler obstruction is a bi-Lipschitz invariant. For hypersurfaces with isolated singularities, we show that the Euler obstruction is preserved under non-degeneracy conditions. These results contribute to the understanding of metric invariants in complex analytic geometry.
	\end{abstract}
	\maketitle
	\tableofcontents
	
	\section*{Introduction}
	
	Bi-Lipschitz invariants have become a central focus in singularity theory, offering insights into the geometric and metric properties of analytic varieties beyond their topological and algebraic classifications. Many fundamental invariants, such  as multiplicity, order, and the \L ojasiewicz exponent, have been extensively studied under bi-Lipschitz equivalence. A central question in this context is the metric version of Zariski's multiplicity conjecture, which asks whether multiplicity is preserved under bi-Lipschitz homeomorphisms. Significant contributions by many authors, including L. Birbrair, J. F. Bobadilla, C. Bivià-Ausina, A. Fernandes, T. Fukui, Z. Jelonek, M. Pe Pereira, and J. E. Sampaio, among others, have advanced our understanding of this problem, revealing deep connections between metric geometry and algebraic invariants. Despite these developments, the behavior of more subtle invariants under bi-Lipschitz equivalence remains less understood.
	
	In this work, we investigate the bi-Lipschitz invariance of two fundamental local invariants in singularity theory: the Łojasiewicz exponent and the local Euler obstruction. Based on the work of Bivià-Ausina and Fukui \cite{carlesfukui}, whose arguments also underpin our proofs, we establish algebraic conditions ensuring the preservation of these invariants. While their results are formulated for ideals in $\mathcal{O}_n$, we extend the setting to ideals in $\mathcal{O}_{X(S)}$, where $(X(S),0)$ is a germ of an affine toric variety associated with a semigroup $S$. Within this framework, we adapt their techniques and derive new conditions under which these invariants remain unchanged under bi-Lipschitz equivalence.

	The status of the local Euler obstruction remains more elusive. This invariant, introduced by MacPherson in the context of Chern classes for singular spaces, is known to be analytic but not a topological invariant. In this paper, we provide a partial answer to the open question of its bi-Lipschitz invariance, demonstrating that under some algebraic conditions for hypersurfaces with isolated singularities, the local Euler obstruction is indeed preserved. This result contributes positively to the broader understanding of metric invariants in complex analytic geometry.
	
	The paper is organized as follows. Section~1 introduces the foundational definitions and notations concerning bi-Lipschitz equivalence, the Łojasiewicz exponent, and the local Euler obstruction and toric varieties. Section~2 extends the bi-Lipschitz invariance of the Łojasiewicz exponent to ideals in $\mathcal{O}_{X(S)}.$ Finally, Section~3 addresses the bi-Lipschitz invariance of the local Euler obstruction for hypersurfaces with isolated singularities under some algebraic assumptions.

		\section{Preliminary Notions and Results}
		
		We fix the notational conventions used throughout this work. Let $(X,x_0)$ be a germ of analytic space in $\C^n$. Consider two real-valued function germs $ a,b:(X, x_0) \to \mathbb{R}$. The relation $a(x) \lesssim b(x)$ near $x_0$ means that there exists a constant $ C > 0 $ and an open neighborhood $U$ of $ x_0$ in $X$ such that
		$$a(x) \leq C \cdot b(x), \quad \text{for all } x \in U.$$
		
		Furthermore, we write  $a(x) \sim b(x)$ near $ x_0$ whenever both $a(x) \lesssim b(x)$ and $b(x) \lesssim a(x)$ hold in some neighborhood of $x_0$.
		
		For any $n$-tuple $ x = (x_1, \ldots, x_n) \in \mathbb{C}^n$, we denote its norm by
		$$\|x\| = \sqrt{|x_1|^2 + \cdots + |x_n|^2}.$$
		
		\subsection{Bi- Lipschitz equivalences}
		
		We start by recalling the definition of a bi-Lipschitz map.
		
		\begin{definition} A map germ $f : (X,0) \to (Y,0)$ is Lipschitz if 
			$$
			\|f(x)- f(x')\| \lesssim \,  \|x- x'\| \text{ near }0.
			$$
			A bi-Lipschitz homeomorphism between $(X,0)$ and $(Y,0)$ is a Lipschitz homeomorphism $h : (X,0) \to (Y,0)$ whose inverse $h^{-1} : (Y,0) \to (X,0)$ is also Lipschitz. Moreover, two subsets $X_1$ and $X_2$ of $(X,0)$ are bi-Lipschitz equivalent if there is a bi-Lipschitz homeomorphism $\varphi : (X, 0) \to (X, 0)$ so that $\varphi(X_1) = X_2$.
		\end{definition}

		Two function germs $f,g : (X, 0) \to (\C^p, 0)$ are called bi-Lipschitz $\mathcal{R}$-equivalent if there exists a bi-Lipschitz map germ $\phi: (X, 0)\to  (X, 0)$ such that 
		$$f = g \circ \phi.$$
		
		Similarly, $f$ and $g$ are \emph{bi-Lipschitz $\mathcal{A}$-equivalent} if there exist bi-Lipschitz homeomorphisms $\varphi : (X,0) \to (X,0)$ and $\phi : (\C^p,0) \to (\C^p,0)$ such that
		$$\phi \circ f = g \circ \varphi.$$
		
		We also say that $f$ and $g$ are \emph{bi-Lipschitz $\mathcal{K}^*$-equivalent} if there is a bi-Lipschitz homeomorphism $\varphi : (X,0) \to (X,0)$ and a map $A : (X,0) \to GL(\mathbb{C}^p)$ such that $A$ and $A^{-1}$ are Lipschitz and     $$ A(x)f(x) = g(\varphi(x))$$
		for all $x$ in some open neighborhood of $0 \in X$, where $GL(\mathbb{C}^p)$ denotes the set of all $p\times p$ complex matrices with nonzero determinant.
		
		In particular, bi-Lipschitz $\mathcal{R}$-equivalence implies both bi-Lipschitz $\mathcal{A}$-equivalence and bi-Lipschitz $\mathcal{K}^\ast$-equivalence.
		
		The notion of \emph{bi-Lipschitz equivalence of ideals}, introduced by Bivià-Ausina and Fukui in~\cite{carlesfukui}, is originally defined for ideals in $\mathcal{O}_n$, by comparing their generators via a bi-Lipschitz homeomorphism $\varphi : (\mathbb{C}^n, 0) \to (\mathbb{C}^n, 0)$. We extend this notion to ideals in $\mathcal{O}_X$, where $X$ denotes a germ of analytic space in $\mathbb{C}^n$, and $\mathcal{O}_X$ denotes the ring of germs of analytic functions $(X,0)\to (\C,0)$.
		
		\begin{definition}
			Let $I$ and $J$ be ideals of $\mathcal{O}_X$. We say that $I$ and $J$ are bi-Lipschitz equivalent if there exist families of functions $f_1, \dots, f_p$ and $g_1, \dots, g_q$ in $\mathcal{O}_X$ such that:
			\begin{enumerate}
				\item $\langle f_1, \dots, f_p \rangle \subseteq I$ and $\overline{\langle f_1, \dots, f_p \rangle} = \overline{I}$; 
				\item $\langle g_1, \dots, g_q \rangle \subseteq J$ and $\overline{\langle g_1, \dots, g_q \rangle} = \overline{J}$;
				\item There exists a bi-Lipschitz homeomorphism $\varphi : (X, 0) \to (X, 0)$ such that
				$$\left\| (f_1(x), \dots, f_p(x)) \right\| \sim \left\| (g_1(\varphi(x)), \dots, g_q(\varphi(x))) \right\| \quad \text{near } 0,$$
			\end{enumerate}
			where $\overline{I}$ denotes the closure of the ideal $I$.
		\end{definition}
		
		It follows directly from the definition that bi-Lipschitz equivalence between two ideals implies bi-Lipschitz equivalence between their respective zero sets. Moreover, if two germs are bi-Lipschitz $\mathcal{R}$-equivalent, then their zero sets are bi-Lipschitz equivalent.
		
		\subsection{ \L ojasiewicz exponent of ideals}
		The \L ojasiewicz exponent is a fundamental numerical invariant that quantifies the comparative vanishing order of two ideals near a singular point. As shown in \cite{L-T}, this exponent can be characterized algebraically through inclusions between powers of ideals and their integral closures. In this subsection, we define the Łojasiewicz exponent $\mathcal{L}_J(I)$ and explore its properties, highlighting its connections with other local invariants such as the Hilbert–Samuel multiplicity and the order of an ideal.
		
		\begin{definition}
			Let $X$ be a reduced equidimensional complex analytic space. Let $I := \langle f_1, \dots, f_p \rangle$ and $J := \langle g_1, \dots, g_q \rangle$ be two ideals in $\mathcal{O}_X$.
			Let us consider the maps 
			$$f = (f_1, \dots, f_p) : (X, 0) \to (\mathbb{C}^p, 0)$$
			$$g = (g_1, \dots, g_q) : (X, 0) \to (\mathbb{C}^q, 0).$$
			
			The \L{}ojasiewicz exponent of $I$ with respect to $J$, denoted by $\mathcal{L}_J(I)$, is the infimum of the set
			\begin{equation*}
				\{\alpha \in \mathbb{R}_{\geq 0} : \|g(x)\|^\alpha \lesssim \|f(x)\| \text{ near } 0\}.
			\end{equation*}
			
			If this set is empty, we say that $\mathcal{L}_J(I) = \infty$. 
		\end{definition}
		
		It is worth noting that $\mathcal{L}_J(I) < \infty$ whenever $I \subseteq \sqrt{J}$. Furthermore, according to \cite{L-T}, if $\mathcal{L}_J(I) < \infty$, then it's a rational positive number.\\
		
		Consider two ideals $I, J \subseteq \mathcal{O}_{X}$. By \cite{L-T}, the \L{}ojasiewicz exponent of $I$ with respect to $J$ is given by
		\begin{equation*}
			\mathcal{L}_J(I) = \min \left\{ \frac{a}{b} \ \middle| \ a, b \in \mathbb{Z}_{\geq 1}, \quad J^a \subseteq \overline{I^b} \right\}.
		\end{equation*}

		Let us assume that the ideal  $I \subseteq \mathcal{O}_X$ has finite colength, that is, the quotient $\mathcal{O}_{X}/I$ is a finite-dimensional $\mathbb{C}$-vector space, since $\mathcal{O}_{X}$ is a $\mathbb{C}$-algebra with residue field $\mathbb{C}$.
		
		In this context, when $J = \mathfrak{m}$ is the maximal ideal of $\mathcal{O}_X$, we denote the number $ \mathcal{L}_J(I) $ by $ \mathcal{L}_0(I) $. That is,
		\begin{equation*}
			\mathcal{L}_0(I) = \inf \left\{ \frac{a}{b}\in \mathbb{Z}_{\geq 1} : \quad \mathfrak{m}^a \subseteq \overline{I^b} \right\}.
		\end{equation*}
		
		The quantity $\mathcal{L}_0(I)$ is called the \emph{Łojasiewicz exponent} of the ideal $I$.
		
		We recall that the \emph{order} of a function $f \in \mathcal{O}_{X}$, denoted by $\operatorname{ord}(f)$, is defined as the largest integer $r \in \mathbb{Z}_{\geq 1}$ such that $f \in \mathfrak{m}^r$. Accordingly, the order of an ideal $I \subset \mathcal{O}_X$ is defined as $ \max \{r: I \subseteq \mathfrak{m}^r\}$.
		
		The Hilbert–Samuel multiplicity of a finite-codimensional ideal $I$ in $\mathcal{O}_X$ is defined as  
		$$e(I) = \lim_{k \to \infty} \frac{n!}{k^n} \, \ell(\mathcal{O}_X / I^k \mathcal{O}_X),$$  
		where $\ell(\mathcal{O}_X / I^k \mathcal{O}_X)$ denotes the colength of $I^k$ (see \cite{Matsumura}). There is a relationship between the Łojasiewicz exponent, the multiplicity of an ideal, and its order. This relation follows from the inequality $e(\mathfrak{m}^a) \geq e(\overline{I^b}) = e(I^b)$ (see \cite{Gaffney1996}), which implies that $a^n \geq b^n e(I)$, and consequently,  $\mathcal{L}_0(I)^n \geq e(I).$ Furthermore, the inclusion $I \subseteq \mathfrak{m}^{\operatorname{ord}(I)}$ yields  $e(I) \geq e(\mathfrak{m}^{\operatorname{ord}(I)}) = \operatorname{ord}(I)^n e(\mathfrak{m}).$  Therefore, we obtain the following inequality:
		\begin{equation}\label{eq1}
			\mathcal{L}_0(I)^n \geq e(I) \geq \operatorname{ord}(I)^n.
		\end{equation}
		
		\subsection{The local Euler obstruction}

		The local Euler obstruction was defined by MacPherson in 
		\cite{MacPherson} as a tool to prove the conjecture about the existence and 
		unicity of the Chern classes in the singular case. Since then, it has been extensively investigated by many
		authors such as Brasselet and Schwartz \cite{BS}, L\^{e} and Teissier \cite{LT}, Kashiwara \cite{Kashiwara}, and others.

	Despite its importance, the local Euler obstruction is generally hard to compute from its definition. This has motivated the search for effective formulas that allow one to handle it in concrete situations.
		
		In \cite[Theorem 3.1]{BLS}, Brasselet, L\^e, and Seade proved a formula to compute the local Euler obstruction using generic linear forms. For this purpose, we first present the definition of generic linear forms.
		
			Let $(X,0) \subset (\mathbb{C}^n,0)$ be a pure-dimensional complex analytic subset $X \subset U$ of an open set $U \subset \mathbb{C}^n$. We consider a complex analytic Whitney stratification $\mathcal{V} = \{V_i\}$ of $U$ adapted to $X$ (i.e. $X$ is a union of strata) and we assume that $\{0\}$ is a stratum. We choose a representative $X$ small enough of $(X,0)$ such that $0$ belongs to the closure of all the strata. We write $X= \cup_{i=0}^q V_i$ where $V_0 = \{0\}$ and $V_q = X_{\rm reg}$ is the set of regular points of $X$. We assume that the strata $V_0,\ldots,V_{q-1}$ are connected. Note that the closures $\overline{V_0},\ldots,\overline{V_{q-1}}$ are complex analytic subsets of $U$.
		
		\begin{definition}
			A generic complex linear form (with respect to $X$) is a complex linear form 
			$\ell : U \to \mathbb{C}$ such that $0 \in \ell^{-1}(0)$ and $\ker(\ell)$ is 
			transversal to all limits of tangent spaces $\{T_{x_n}V_i\}$, for every stratum 
			$V_i$ and every sequence $\{x_n\} \subset V_i$ converging to $0$.
		\end{definition}
		
		\begin{theorem}\label{BLS}
			Let $(X,0)$ and $\mathcal{V}$ be given as before, then for each generic 
			linear form $l$, there exists $\varepsilon_0$ such that for any $\varepsilon$ with 
			$0<\varepsilon<\varepsilon_0$ and $\delta\neq0$ sufficiently small, the Euler 
			obstruction of $(X,0)$ is equal to 
			\[
			{\rm Eu}_X(0)=\sum^{q}_{i=1}\chi(V_i\cap B_{\varepsilon}\cap l^{-1}(\delta)) \cdot 
			{\rm Eu}_{X}(V_i),
			\]
			where $\chi$ is the Euler characteristic, ${\rm Eu}_{X}(V_i)$ is the Euler 
			obstruction of $X$ at a point of $V_i, \ i=1,\ldots,q$ and 
			$0<|\delta|\ll\varepsilon\ll1$. 
		\end{theorem} 
		
		The local Euler obstruction is known to be an analytic invariant, meaning it remains unchanged under analytic equivalence. However, it is not a topological invariant, as it can vary under homeomorphisms (see \cite{teseNivaldo}). Whether the local Euler obstruction is a bi-Lipschitz invariant remains an open question. In this work, we provide a partial answer to this problem, showing that under certain conditions, the local Euler obstruction is indeed preserved by bi-Lipschitz transformations. Thus, our result offers a positive contribution toward understanding its behavior in the bi-Lipschitz category.
		This leads naturally to the study of polar multiplicities, which are closely linked to the local Euler obstruction and play a central role in its computation and geometric interpretation.

		The next theorem,  proved by L\^e and Teissier, presents a way to compute the local Euler obstruction with the aid of Gonz\'ales-Sprinberg's purely algebraic interpretation of this invariant.
		
		\begin{theorem}[\cite{LT}] \label{lt} Let $X \subset \mathbb{C}^{n+1}$ be an analytic space of dimension $d$ reduced at $0$. Then ${\rm Eu}_{0}(X) = \sum_{k=0}^{d-1}(-1)^{d-k-1} m_k(X)$, where $m_k(X)$ denotes the $k$-th polar multiplicity of $X$ at $0$.
		\end{theorem}
			
		In particular, the first polar multiplicity $m_0(X)$ coincides with the classical notion of multiplicity, which is the central object in the metric version of Zariski’s conjecture. This conjecture asks whether multiplicity is preserved under bi-Lipschitz homeomorphisms. 
		
		\subsection{Toric Varieties}
		
		In this subsection, we recall some fundamental notions on toric varieties that will be used throughout this work. In particular, the combinatorial description via semigroups and convex polyhedra plays an essential role in the analysis of ideals and their invariants, especially in the construction and study of the Newton polyhedra that will appear in the subsequent results. For more references, see, for instance, \cite{Cox}.\\
		
		Let $S \subset  \mathbb{Z}^{n}_{\geq0}$ be a finitely generated semigroup with respect to addition, satisfying $0 \in S$, with $\mathbb{Z}S=\mathbb{Z}^n$. Let $\{\gamma_1,\dots,\gamma_r\} \subset \mathbb{Z}^{n}_{\geq0}$ be a set of generators of $S$. This set induces a homomorphism of groups $\pi_{S} : \mathbb{Z}^r \to \mathbb{Z}^n$ given by
		$$\pi_{S}(\alpha_1,\dots,\alpha_r) = \alpha_1 \gamma_1 + \dots + \alpha_r \gamma_r.$$
		
		For each $\alpha = (\alpha_1, \dots, \alpha_r) \in Ker(\pi_{S})$ we set 
		$$
		{\alpha}_+ = \displaystyle \sum_{\alpha_i > 0} \alpha_i e_i \ \ \ \ \text{and} \ \ \ \ {\alpha}_{-} = -\displaystyle \sum_{\alpha_i < 0} \alpha_i e_i,
		$$
		where $e_1, \dots, e_r$ are the elements of the standard basis of $\mathbb{R}^r$. Let us observe that $\alpha = \alpha_{+} - \alpha_{-}$ and $\alpha_{+}, \alpha_{-} \in \mathbb{N}^r$, where $\mathbb{N}=\{0,1,2,\ldots\}$.

		Consider the ideal 
		\begin{equation*}\label{nucleo}
			I_S = \langle x^{\alpha_{+}} - x^{\alpha_{-}}; \ \ \alpha \in Ker(\pi_{S}) \rangle \subset \mathbb{C}[x_1,\dots,x_r],
		\end{equation*}
		where $x^{\beta} = x_1^{\beta_1} \dots x_r^{\beta_r}$ with $\beta = (\beta_1, \dots, \beta_r) \in \mathbb{N}^r$. It is possible to prove that $I_S$ is a prime ideal (see Propositions $1.1.9$ and $1.1.11$ in \cite{Cox}). The $n$-dimensional affine toric variety $X(S) \subset \mathbb{C}^r$ generated by $S$ is the irreducible  variety $\mathcal{V}(I_S)$ defined by the zero set of the ideal $I_S$ as in (\ref{nucleo}) , which is not necessarily normal. \\
		
		In the following, we introduce the concepts of cone and dual cone, which are studied in convex geometry. Those elements will be necessary for defining the concept of Newton polyhedron, which we will use in this work.
		\begin{definition}
			Let $\{\gamma_{1},\dots,\gamma_{r}\}\subset \Z^{n}_{\geq 0}$ be a finite generator set of $S$. The convex polyhedral cone associated with $S$ in $\R^{n}$ is the set 
			$$ cone(S)=\Bigg\{ \sum_{i=1}^{r} \lambda_{i}\gamma_{i} \ \vert \ \lambda_{i} \in \R \ , \ \lambda_{i}\geq 0 \Bigg\}.$$
			The vectors $\gamma_{1},\dots,\gamma_{r}$ are called generators of the  $cone(S)$. Also, set $cone(\emptyset)=\{0\}$.
		\end{definition}
		
		In what follows, we will consider that  $dim( cone(S))=n$ and that $cone(S)$ is strongly convex, i.e. $cone(S) \cap (- cone(S)) = \{0\}$. Let $(\R^{n})^{\ast}$ be the dual space of $\R^{n}$. To each cone, we associate the dual cone, defined by $$\coned{S}=\{ u \in (\R^{n})^{*} : \langle u,v \rangle \geq 0 , \forall v \in cone(S)\}.$$
		
		Let $g \in \OX{S}$, i.e., $g=h\vert_{X(S)}$, where $h: \C^{r} \to \C$ is an analytic function such that $h=\sum_{k} a_{k}x^{k}$ is its Taylor expansion. The support of $g \in \OX{S}$ is defined by the set
		$$supp (g) = \Bigg\{\nu(k): \ \ \sum_{\nu(k') = \nu(k)} a_{k'}\neq0 \ \  \text{ and }k=(k_1,\dots,k_r)\Bigg\}, $$
		where $\nu(k)=k_1\gamma_1+\cdots+k_r\gamma_r \in S$ and $\gamma_1,\dots, \gamma_r$ are generators of $S$. If $G \subseteq \OX{S}$ is a subset of the ring $\OX{S}$, we define the support of $G$ as \begin{equation}
			\label{Supp-Ideal}
			supp(G)=\bigcup_{g \in G}supp(g).
		\end{equation}
		
		\begin{definition}
			Let $G$ be a subset of the ring $\OX{S}$. The convex hull of $$\displaystyle\bigcup_{\alpha\in supp(G)}(\alpha+ \operatorname{cone}(S)) \subset \operatorname{cone}(S)$$ is called the \emph{Newton polyhedron} of $G$ and will be denoted by $\Gamma_+(G)$. 
		\end{definition}
		
		If $I \subseteq \OX{S}$ is the  ideal generated by $G$,  then, by the definition of convex hull and the properties of the support, it follows that $\Poli(I)=\Poli(G)$. Moreover, if $G = \{g_{1},\dots,g_{l}\}$, then $\Poli(I)$ is equal to the convex hull of $\Poli(g_{1}) \cup \cdots\cup \Poli(g_{l})$. Therefore, $\Poli(I)$ is independent of the chosen generating system.\\
		
		Now, given a compact subset $A \subseteq cone(S)$, we denote by
		$$g_{A}:= \sum_{\nu(k) \in A\cap supp(g)} a_{k} x^{k}.$$
		We set $g_{A} = 0$ whenever $A \cap supp(g) = \emptyset$.
		
		Let $g$ be a representative of the germ $g \in \mathcal{O}_{X(S)}$ such that $0 \notin supp(g)$. Then, we associate to $g$ a polynomial $L(g) \in \On$ given by
		$$L(g)= \sum_{v \in supp(g)} a_v z^v,$$
		where $a_v=\displaystyle\sum_{\nu(k') = \nu(k)} a_{k'}\neq0 $ and $v=\nu(k) \in S$.
		\begin{definition}
			Let $G=\{g_{1},\dots,g_{l}\} \subseteq \OX{S}$. We say that $G$ is non-degenerate if, for each compact face $\Delta$ of $\Poli(G)$, the equations $$L\big((g_1)_{\Delta}\big)= \cdots= L\big((g_l)_{\Delta}\big)=0$$
			have no common solution in $(\C^{\ast})^{s}$.
			
			An ideal $I\subseteq \OX{S}$ is non-degenerate if $I$ admits a non-degenerate system of generators.
			
		\end{definition}
		
		\begin{remark}
			Let $\{e_{1},\ldots,e_{s}\}$ be a canonical basis of $\mathbb{C}^s$. If $S=\langle e_{1},\ldots,e_{s} \rangle$, then the toric variety $X(S)=\C^{s}$. In this case, we have $\operatorname{cone}(S)=\R^{s}_{\geq 0}$. Thus, the previous definition can be seen as a generalization of the Newton non-degeneracy condition for germs in $\mathcal{O}_{s}$ presented by Saia in \cite{Saia}.
		\end{remark}

		Given an ideal $I$ of $\mathcal{O}_{X(S)}$, we denote by $I^{\circ}$ the ideal of $\OX{S}$ generated by the monomials $x^{k}$ such that $k_{1}\gamma_{1}+\cdots+k_{r}\gamma_{r} \in \Poli(I)$, where $k=(k_1,\ldots, k_r)$. The following result provides a characterization of non-degeneracy condition for ideals in $\mathcal{O}_{X(S)}$.
		
		\begin{theorem}[see \cite{ATT}]\label{ICideals}
			Let $I$ be an ideal in $\mathcal{O}_{X(S)}$. Then $I$ is non-degenerate if and only if its integral closure $\overline{I}$ is equal to the ideal $I^{\circ}$. 
		\end{theorem}
		
		\section{The Bi-Lipschitz Invariance of the \L ojasiewicz Exponent}
		
		The study of bi-Lipschitz invariants plays a central role in singularity theory, as these invariants capture geometric properties of analytic varieties beyond their purely algebraic structure. This section generalizes the work of Bivia-Ausina and Fukui \cite{carlesfukui}, extending their results on the bi-Lipschitz invariance of the order and \L ojasiewicz exponent from ideals in $\mathcal{O}_n$ to ideals in $\mathcal{O}_{X(S)}$, where $(X(S),0)$ is a germ of an affine toric variety associated with the semigroup $S$. As a first step, we establish a general result on the bi-Lipschitz equivalence of ideals, which will serve as a fundamental tool for the subsequent analysis of Jacobian ideals and their integral closures.\\

		The proofs of Theorems \ref{teo-ord}, \ref{Jfbl}, and Corollary \ref{BLclosure} are inspired by the arguments presented in the work of Bivià-Ausina and Fukui \cite{carlesfukui}.

		\begin{theorem}\label{teo-ord}
			Let $I$ and $J$ be ideals of $\mathcal{O}_{X(S)}$.  
			If $I$ and $J$ are bi-Lipschitz equivalent, then $\operatorname{ord}(I) = \operatorname{ord}(J)$.  
			Moreover, if $I$ and $J$ have finite colength, then $\mathcal{L}_0(I) = \mathcal{L}_0(J)$.
		\end{theorem}
		
		\begin{proof}
			Since $I$ and $J$ are bi-Lipschitz equivalent, there exist germs of analytic maps  
			$$f = (f_1, \dots, f_p) : (X(S), 0) \to (\mathbb{C}^p, 0) $$
			and  
			$$g = (g_1, \dots, g_q) : (X(S), 0) \to (\mathbb{C}^q, 0)$$
			such that  
			$$\overline{I} = \overline{\langle f_1, \dots, f_p \rangle}\quad \text{and} \quad \overline{J} = \overline{\langle g_1, \dots, g_q \rangle},$$
			and there exists a bi-Lipschitz homeomorphism  $ \varphi : (X(S), 0) \to (X(S), 0)$ satisfying  
			$$\|g(\varphi(x))\| \sim \|f(x)\| \quad \text{near } 0.$$
			By symmetry, it suffices to prove that $\mathcal{L}_0(I) \leq \mathcal{L}_0(J)$ and $\operatorname{ord}(I) \leq \operatorname{ord}(J)$. Let $\mathcal{L}_0(J) = \theta \in \mathbb{R}_{\geq 0}$ be such that   $$\|x\|^\theta \lesssim \|g(x)\| \quad \text{ near }0.$$
			Then
			$$\|x\|^\theta \sim \|\varphi(x)\|^\theta \lesssim \|g(\varphi(x))\| \sim \|f(x)\| \quad \text{ near }0,$$
			and it follows that $\mathcal{L}_0(I) \leq \mathcal{L}_0(J)$. Moreover, we have
			$$ \operatorname{ord}(J) = \max \{s : J \subseteq \mathfrak{m}^s\}  =\max \{s : J \subseteq \overline{\mathfrak{m}^s}\} = \max \{s : \|g(x)\| \lesssim \|x\|^s \text{ near } 0\},$$
			where $\mathfrak{m}$ is the maximal ideal of $\mathcal{O}_{X(S)}$.
			
			If $\|f(x)\| \lesssim \|x\|^{ord(I)} \text{ near } 0,$ then
			$$\|g(x)\| \sim \|f(\varphi(x))\| \lesssim\|\varphi(x)\|^{ord(I)} \sim \|x\|^{ord(I)} \text{ near } 0,$$
			and we obtain  $\operatorname{ord}(I) \leq \max \{s : \|g(x)\| \lesssim \|x\|^s \text{ near } 0\} =ord(J).$
		\end{proof}
		
		\begin{corollary}\label{BLclosure}
			Let $I$ and $J$ be ideals of $\mathcal{O}_{X(S)}$ with finite colength. 
			Let us suppose that $\overline{I} = \mathfrak{m}^{\mathrm{ord}(I)}$. Then, $I$ and $J$ are bi-Lipschitz equivalent if and only if $\overline{I} = \overline{J}$.
		\end{corollary}
		\begin{proof}
			The “if” part is immediate. If $I$ and $J$ are bi-Lipschitz equivalent, then by Theorem~\ref{teo-ord} we have $\operatorname{ord}(I) = \operatorname{ord}(J)$ and $\mathcal{L}_0(I) = \mathcal{L}_0(J)$. Since $\overline{I} = \mathfrak{m}^{\operatorname{ord}(I)}$, it follows that $\mathcal{L}_0(I) \le \operatorname{ord}(I)$. Moreover, $I \subseteq \mathfrak{m}^{\operatorname{ord}(I)}$ implies $I^q \subseteq \mathfrak{m}^{q \cdot \operatorname{ord}(I)}$ and hence $\overline{I^q} \subseteq \mathfrak{m}^{q \cdot \operatorname{ord}(I)}$ for all $q \in \mathbb{N}$. If $\mathfrak{m}^p \subseteq \overline{I}^q$, then $p \ge q \cdot \operatorname{ord}(I)$, so $$\operatorname{ord}(I) \le \inf\{ p/q : \mathfrak{m}^p \subseteq \overline{I}^q \} = \mathcal{L}_0(I).$$ Therefore, $\operatorname{ord}(I) = \mathcal{L}_0(I)$.
			
			Since $J \subseteq \mathfrak{m}^{\operatorname{ord}(J)}$, by (\ref{eq1}) we obtain
			$$ e(\mathfrak{m})\,\operatorname{ord}(I)^n 
			\ \geq\ 
			\operatorname{ord}(I)^n 
			= \mathcal{L}_0(I)^n 
			= \mathcal{L}_0(J)^n 
			\ \geq\ e(J) 
			\ \geq\ e(\mathfrak{m})\,\operatorname{ord}(J)^n 
			= e(\mathfrak{m})\,\operatorname{ord}(I)^n.
			$$
			Hence $e(J)=e(\mathfrak{m}^{\operatorname{ord}(I)})$. By the Rees Multiplicity Theorem (see \cite{hunekeswanson}), $ \mathfrak{m}^{\operatorname{ord}(I)} \subseteq \overline{J}$. Therefore, $\overline{J}=\overline{I}$.
		\end{proof}
		The following example illustrates a situation where the hypothesis $\overline{I} = \mathfrak{m}^{\mathrm{ord}(I)}$ of Corollary \ref{BLclosure} is satisfied.
		
		\begin{example}
			Let $I \subset \mathcal{O}_{X(S)}$ be a non-degenerate ideal, where $X(S)$ is a normal toric variety with $S = \langle b_1, \dots, b_r \rangle$. If $\Gamma_+(I) = \Gamma_+(\mathfrak{m}^{\mathrm{ord}(I)})$, then $\overline{I} = \mathfrak{m}^{\mathrm{ord}(I)}$. Indeed, by Theorem \ref{ICideals}, the integral closure is given by
			$$\overline{I} = \langle x^k : k_1b_1 + \cdots + k_rb_r \in \Gamma_+(I) \rangle.$$
			The hypothesis $\Gamma_+(I) = \Gamma_+(\mathfrak{m}^{\mathrm{ord}(I)})$ implies that
			$$\overline{I} = \langle x^k : k_1b_1 + \cdots + k_rb_r \in \Gamma_+(\mathfrak{m}^{\mathrm{ord}(I)}) \rangle = \overline{\mathfrak{m}^{\mathrm{ord}(I)}}.$$
			Since $X(S)$ is normal, the ideal $\mathfrak{m}^{\mathrm{ord}(I)}$ is integrally closed, and we conclude that $\overline{I} = \mathfrak{m}^{\mathrm{ord}(I)}$.
		\end{example}
		
		Before proceeding, we recall that for a germ $f\in \mathcal{O}_{X(S)}$ the Jacobian ideal is defined by $$J(f)=\left\langle \frac{\partial f}{\partial x_1},\dots,\frac{\partial f}{\partial x_n}\right\rangle \subseteq \mathcal{O}_{X(S)}.$$
		In the case of an isolated singularity at the origin, this ideal has finite colength.

		\begin{theorem}\label{Jfbl}
			Let $f, g \in \mathcal{O}_{X(S)}$ be germs with an isolated singularity at the origin. Suppose that $f$ and $g$ are bi-Lipschitz $\mathcal{A}$-equivalent or bi-Lipschitz $\mathcal{K}^*$-equivalent.
			Then the Jacobian ideals $J(f)$ and $J(g)$ are bi-Lipschitz equivalent. In particular, $\operatorname{ord}(f) = \operatorname{ord}(g)$ and $\mathcal{L}_0(J(f)) = \mathcal{L}_0(J(g))$.
		\end{theorem}
		\begin{proof}	
				Suppose that $f$ and $g$ are bi-Lipschitz $\mathcal A$-equivalent, that is, there exist bi-Lipschitz homeomorphisms
				$\varphi:(X(S),0)\to (X(S),0)$ and $\phi:(\C,0)\to (\C,0)$
				such that $\phi\circ f = g\circ\varphi$.
				
				We identify $(X(S),0) \subset \C^n$ with an analytic subvariety of $\R^{2n}$; in this setting, the regular part $X(S)_{reg}$ is a smooth real manifold of dimension $2\dim_{\C}X(S)$.
				
				By Rademacher’s theorem, $\phi$ and $\varphi$ are differentiable almost everywhere on $X(S)_{reg}$ and on $\C$. Thus, the derivatives $D\phi(x)$ and $D\varphi(x)$ exist almost everywhere on $X(S)_{reg}$. Moreover, since $\varphi$ and $\varphi^{-1}$ are bi-Lipschitz, their derivatives (as well as those of $\phi$ and $\phi^{-1}$) exist almost everywhere and are essentially bounded, with bounded inverses as well.
				
				For every $x \in X(S)_{reg}$ where the derivatives exist, the chain rule yields
				\begin{equation}\label{chainrule1}
					\nabla(g\circ \varphi)(x) \;=\; D\varphi(x)^T \, \nabla g(\varphi(x)),
				\end{equation}
				where $\nabla g$ denotes the real gradient of $g$ restricted to $X(S)_{reg}$. On the other hand,
				\begin{equation}\label{chainrule2}
					\nabla(\phi\circ f)(x) \;=\; D\phi(f(x))\,\nabla f(x).
				\end{equation}
				
				Comparing \eqref{chainrule1} and \eqref{chainrule2} and taking norms, we obtain
				$$\|D\varphi(x)^T \nabla g(\varphi(x))\|
				= \|D\phi(f(x)) \nabla f(x)\|
				\leq \|D\phi(f(x))\| \, \|\nabla f(x)\|
				\leq M \|\nabla f(x)\|.$$
				
				Since $D\phi(f(x))$ is invertible with $\|D\phi(f(x))^{-1}\|\leq L$, it follows that
				$$\|\nabla g(\varphi(x))\| \leq \|D\phi(f(x))^{-T}\| \, \|D\phi(f(x)) \nabla f(x)\|\leq LM \|\nabla f(x)\|.$$
				
				Similarly, using $\varphi^{-1}$ and $\phi^{-1}$, we obtain
				$$\|\nabla f(x)\| \;\leq\; LM \|\nabla g(\varphi(x))\|.$$
				
				Therefore, there exist constants $C_1, C_2 > 0$ such that, for almost every $x \in X(S)_{reg}$,
				$$C_1 \|\nabla f(x)\| \;\leq\; \|\nabla g(\varphi(x))\| \;\leq\; C_2 \|\nabla f(x)\|.$$
				
				We now extend this inequality to all points near the origin. Note that the singular locus of $X(S)$, $X(S)_{\mathrm{sing}}$, is closed, of real codimension at least $2$, hence of measure zero, and that $\nabla f$ and $\nabla g$ are continuous on $X(S)$. Thus, an inequality holding almost everywhere on $X(S)_{reg}$ extends, by density and continuity, to all points of $X(S)$.
				
				More precisely,  given $x_0 \in X(S)_{reg}$, there exists a sequence $x_n \in X(S)_{reg}$ with $x_n \to x_0$ such that the inequality holds for all $x_n$. Taking the limit as $n \to \infty$ and using the continuity of $\nabla f$ and $\nabla g$, the inequality also holds at $x_0$. At the origin $0 \in X(S)$, since $f$ and $g$ have isolated singularities, $\nabla f(0) = \nabla g(0) = 0$, which is compatible with the inequality.
				
				We conclude that the inequality holds in a neighborhood $V \subset X(S)$ of the origin. Hence, $J(f)$ and $J(g)$ are bi-Lipschitz equivalent.
				
				Now suppose that $f$ and $g$ are bi-Lipschitz $\mathcal K^*$-equivalent. 
				That is, there exist a bi-Lipschitz homeomorphism $\varphi:(X(S),0)\to (X(S),0)$ and a Lipschitz function $A:(X(S),0)\to\C^*$, with Lipschitz inverse $A^{-1}$, such that $g\circ \varphi = A\cdot f$ in a neighborhood of the origin in $X(S)$. Then we have that
				$$\begin{aligned}
					\|\big(\nabla g\big)(\varphi(x))\| 
					&\lesssim \|\big(\nabla g\big)(\varphi(x)) D\varphi(x)\| 
					&& \text{(since $\varphi$ is bi-Lipschitz)} \\[6pt]
					&= \|\nabla (g \circ \varphi)(x)\| 
					&& \text{(by \ref{chainrule1} )} \\[6pt]
					&= \|\nabla A(x) f(x) + A(x)\nabla f(x)\| 
					&& \text{(since $g(\varphi(x)) = A(x) f(x)$)} \\[6pt]
					&\leq \|\nabla A(x)\|\,|f(x)| + \|A(x)\| \,\|\nabla f(x)\| \\[6pt]
					&\lesssim |f(x)| + \|\nabla f(x)\| 
					&& \text{(since $A(x)$ is Lipschitz)} \\[6pt]
					&\lesssim \|x\|\,\|\nabla f(x)\| + \|\nabla f(x)\| 
					&& \text{(since $|f(x)| \lesssim \|x\|\|\nabla f(x)\|$)} \\[6pt]
					&\lesssim \|\nabla f(x)\|,
				\end{aligned}$$
				almost everywhere. Similarly, we have 
				$$\|(\nabla f)(\varphi^{-1}(x))\| \lesssim \|\nabla g(x)\|$$
				near $0$, and hence we obtain that the ideals $J(f)$ and $J(g)$ are bi-Lipschitz equivalent.
			\end{proof}
			
			\begin{corollary}\label{closurejacobian}
				Let $f \in \mathcal{O}_{X(S)}$ be such that $\overline{J(f)} = \mathfrak{m}^{\operatorname{ord}(f)-1}$. 
				If $g \in \mathcal{O}_{X(S)}$ is bi-Lipschitz $\mathcal{A}$-equivalent 
				to $f$ or bi-Lipschitz $\mathcal{K}^*$-equivalent to $f$, 
				then $\overline{J(f)}=\overline{J(g)}$.
			\end{corollary}
			\begin{proof}
				It follows from Theorem \ref{teo-ord} and Corollary \ref{BLclosure}.
			\end{proof}
			Before presenting the example below, we briefly recall the notion of a semi–homogeneous function. A germ $f \in \mathcal{O}_{X(S)}$ is called \emph{semi–homogeneous of order $d$} if it admits a decomposition of the form
			$$f = f_d + \text{(higher order terms)},$$
			where $f_d$ is a nonzero homogeneous polynomial of degree $d$. Equivalently, $f$ is semi–quasi–homogeneous  with respect to the weight $(1,\dots,1)$.
			
			\begin{example}
				If $f \in \mathcal O_{n}$ is semi--homogeneous of order $d$ and its homogeneous component $f_d$ has an isolated singularity, then
				$$\overline{J(f)} = \mathfrak{m}^{\,d-1},$$
				where $\mathfrak{m}$ is the maximal ideal of $\mathcal{O}_{n}$. This follows from growth estimates comparing $\nabla f =( \frac{\partial f}{\partial x_1},\dots, \frac{\partial f}{\partial x_n})$ with monomials of degree $d-1$, which ensure that all such monomials belong to the integral closure of the Jacobian ideal.
				
				We can write $$f=f_d+\text{(terms of order $>d$)},$$ where $f_d$ is homogeneous of degree $d$ and has an isolated singularity at $0$.
				
				Since each partial derivative $\partial f/\partial x_i$ has order $\ge d-1$, we obtain $J(f)\subseteq\mathfrak m^{d-1}$ and hence $\overline{J(f)}\subseteq \mathfrak m^{d-1}$. For the reverse inclusion, let $q(x)$ be any monomial of degree $d-1$. Writing $x=ru$ with $r=\|x\|$ and $u\in S^{2n-1}$, we have $$\nabla f_d(x)=r^{d-1}\nabla f_d(u) \quad \text{ and } \quad q(x)=r^{d-1}q(u).$$ 
				
				Since $f_d$ has an isolated singularity, $\nabla f_d(u)\neq0$ for all $u\in S^{2n-1}$; by Weierstrass' Theorem there exist constants $m_0>0$ such that $$\|\nabla f_d(x)\|\ge m_0\|x\|^{d-1}.$$
				
				Taking $M_0:= \max_{u \in S^{2n-1}}  |q(u)| < \infty$, we have $|q(x)|\le M_0\|x\|^{d-1}.$
				
				Moreover, $\nabla f(x)=\nabla f_d(x)+R(x)$, where $R(x)$ are the derivatives of the terms of order $>d$. Note that $\operatorname{ord}R\ge d$, hence $\|R(x)\|\le \tfrac{m_0}{2}\|x\|^{d-1}$ in a neighborhood of $0$, and consequently $$\|\nabla f(x)\|\ge \tfrac{m_0}{2}\|x\|^{d-1}$$
				in a neighborhood of $0$. Hence, $$| q(x)|  \leq  \frac{2M_0}{m_0} \|\nabla f(x)\|$$ 
				in a neighborhood of $0$.
				Therefore, in the same neighborhood, we get 
				$$|q(x)| \leq C \max_{1\leq i\leq n}\Big\{\frac{\partial f}{\partial x_i}(x)\Big\}$$
				for some constant $C>0$, which by the Growth Condition implies $q\in\overline{J(f)}$. Since this holds for every monomial of degree $d-1$, we conclude $\mathfrak m^{d-1}\subseteq \overline{J(f)}$, and the equality follows.
			\end{example}
			
			\begin{example}
				Consider the toric variety $X(S)$ associated with the semigroup $S=\{(1,0),(1,1),(1,2)\}$. Let $g = x^a + y^b + z^a + R(x,y,z) \in \mathcal{O}_{X(S)}$, where $b \geq a$ and $R(x,y,z)$ consists of terms of order strictly greater than $a$. The Newton polyhedron of the Jacobian ideal $J(g)$ is depicted in Figure \ref{J}.
				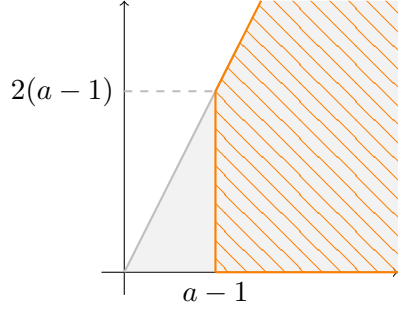
\begin{figure}[H]
					\centering
					\begin{tikzpicture}[scale=0.6]
						\draw[->] (-0.5,0) -- (6,0);
						\draw[->] (0,-0.5) -- (0,6);

						\draw[lightgray, thick] (0,0) -- (2.5,5);
						
						\fill[lightgray, opacity=0.2](0,0) -- (3,6) -- (6,6) --(6,0)  --cycle;
						
						\fill[black] (0,4) node[left] {$2(a-1)$};
						
						\fill[black] (2,0)  node[below] {$a-1$};
						
						\draw[orange, thick] (2,4) -- (3,6);
						\draw[orange, thick] (2,4) -- (2,0);
						\draw[orange, thick] (2,0) -- (6,0);
						
						\pattern[pattern={hatch[hatch size=7pt, hatch linewidth=.2pt, hatch angle=90]} , pattern color=orange](2,4) -- (3,6) -- (6,6) -- (6,0)--(2,0)--cycle;

						\draw[lightgray,  thick, dashed] (0,4) -- (2,4);

					\end{tikzpicture}
					\caption{ Newton polyhedron of $J(g)$.}
					\label{J}
				\end{figure}
				Note that $J(g)$ is non-degenerate, and by Theorem \ref{ICideals} we have 
				\begin{eqnarray*}
					\overline{J(g)} 
					&=& \langle x^a y^b z^c \;:\; a(1,0)+b(1,1)+c(1,2)\in \Gamma(J(g)) \rangle \\
					&=& \langle x^a y^b z^c \;:\; a(1,0)+b(1,1)+c(1,2)\in \Gamma(\mathfrak{m}^{\,a-1}) \rangle \\
					&=& \overline{\mathfrak{m}^{\,a-1}} \;=\; \mathfrak{m}^{\,a-1}.
				\end{eqnarray*}
				
			\end{example}

			\section{The Bi-Lipschitz Invariance of the Euler Obstruction for Hypersurfaces}
			In this section, we study the bi-Lipschitz invariance of the Euler obstruction for hypersurfaces with an isolated singularity defined in a toric variety $X(S)$. Using the formulas of Matsui and Takeuchi \cite{Matsui} for the Euler characteristic via Newton polyhedra, we show that, under the same hypotheses of Corollary \ref{closurejacobian}, this obstruction is preserved under bi-Lipschitz equivalence of ideals.\\

					\begin{theorem}\label{teoeuler}
						Let $f, g \in \mathcal{O}_{X(S)}$ be germs with an isolated singularity at the origin such that $\overline{J(f)} = \mathfrak{m}^{\operatorname{ord}(f)-1}$,  where $\mathfrak{m}$ is the maximal ideal of $\mathcal{O}_{X(S)}$.
						Suppose that $f$ and $g$ are bi-Lipschitz $\mathcal{A}$-equivalent 
						or bi-Lipschitz $\mathcal{K}^*$-equivalent. Then  $$
						Eu_{V(f)}(0)= Eu_{V(g)}(0),
						$$
						where $V(f)$ and $V(g)$ are the hypersurfaces given by the zero sets of $f$ and $g$, respectively.
					\end{theorem}
					\begin{proof} Since $\overline{J(f)} = \mathfrak{m}^{\operatorname{ord}(f)-1}$, we have $\Gamma_+(J(f))=\Gamma_+\!\bigl(\mathfrak{m}^{\operatorname{ord}(f)-1}\bigr).$ Hence,   
						\begin{eqnarray*}
							J(f)^\circ
							&: =& \big\langle x^k : k_1\gamma_1+\cdots+k_r \gamma_r \in \Gamma_+(J(f)) \big\rangle \\
							&=& \big\langle x^k : k_1\gamma_1+\cdots+k_r \gamma_r \in \Gamma_+\!\bigl(\mathfrak{m}^{\operatorname{ord}(f)-1}\bigr) \big\rangle\\
							&=& \bigl(\mathfrak{m}^{\operatorname{ord}(f)-1}\bigr)^\circ .
						\end{eqnarray*}
						Since $\mathfrak{m}^{\operatorname{ord}(f)-1}$ is non-degenerate, Theorem \ref{ICideals} implies that  $\bigl(\mathfrak{m}^{\operatorname{ord}(f)-1}\bigr)^\circ
						= \overline{\mathfrak{m}^{\operatorname{ord}(f)-1}}$.
						Thus $\overline{J(f)} = J(f)^\circ$. Applying  Theorem \ref{ICideals} once more, we conclude that $J(f)$ is a non-degenerate ideal.
						
						By Theorem \ref{Jfbl}, it follows that $J(f)$ and $J(g)$ are bi-Lipschitz equivalent. Since  $\overline{J(f)} =\mathfrak{m}^{\operatorname{ord}(f)-1}$, the Corollary \ref{closurejacobian} implies that $\overline{J(f)} = \overline{J(g)}$.  Hence $\Gamma_{+}(J(f)) = \Gamma_{+}(J(g))$. Note that the equality $\Gamma_{+}(J(f)) = \Gamma_{+}(\mathfrak{m}^{\operatorname{ord}(f)-1})$ implies that $\Gamma_{+}(f)=\Gamma_{+}(\mathfrak{m}^{\operatorname{ord}(f)})$. Applying the same argument to $g$, we conclude that $\Gamma_{+}(f)=\Gamma_{+}(g)$.
						
						Moreover,
						\begin{eqnarray*}
							J(g)^\circ
							&:=& \big\langle x^k : k_1\gamma_1+\cdots+k_r \gamma_r \in \Gamma_+(J(g)) \big\rangle \\
							&=& \big\langle x^k : k_1\gamma_1+\cdots+k_r \gamma_r \in \Gamma_+(J(f)) \big\rangle \\
							&=& \bigl(\mathfrak{m}^{\operatorname{ord}(f)-1}\bigr)^\circ\\
							&=& J(f)^\circ\\
							&=& \overline{J(f)}\\
							&=& \overline{J(g)},
						\end{eqnarray*}
						then, applying Theorem \ref{ICideals} once more, we conclude that $J(g)$ is also an ideal non-degenerate. Thus, $f$ and $g$ are non-degenerate functions in the sense of \cite[Definition $3.2$]{Matsui}.
						
						Now, let $L_f, L_g : \mathbb{C}^r \to \mathbb{C}$ be generic  linear forms with respect to $V(f)$ and $V(g)$, respectively. Without loss of generality, we can assume that
						$$L_g(x) = a_1 x_1 + \cdots + a_r x_r, \quad a_i \neq 0, \quad \text{for all i}$$
						$$L_g(x) = b_1 x_1 + \cdots + b_r x_r, \quad b_i \neq 0, \quad \text{for all i}$$
						so that $\Gamma_{+}(L_f) = \Gamma_{+}(L_g)$.
						
						Consider the following subvarieties of $X(S)$
						\begin{equation}
							V(f,L_f) := \{f= L_f = 0\} \subset
							V(f) := \{f = 0\},
						\end{equation}
						\begin{equation}
							V(g,L_g) := \{g= L_g = 0\} \subset
							V(g) := \{g = 0\}.
						\end{equation}
						Since $f$ and $g$ are non-degenerate functions, for $L_f$ and $L_g$ generic enough, we have $(f,L_f)$ and $(g,L_g)$ as non-degenerate complete intersections, as in \cite[Definition $3.9$]{Matsui} (see also \cite{oka}).
						
						Therefore, by \cite[Theorem 3.12]{Matsui}, the Euler characteristic of the Milnor fiber $(L_f)_0$ of the function $L_f|_{V(f)}: V(f) \to \mathbb{C}$ at the point $0 \in V(f)$ is given as a sum of volumes of the Newton polyhedron associated to the function $f \cdot L_f$. 
						
						Since $\Gamma_{+}(f) = \Gamma_{+}(g)$ and $\Gamma_{+}(L_f) = \Gamma_{+}(L_g)$, it follows that $\Gamma_{+}(f \cdot L_f) = \Gamma_{+}(g \cdot L_g)$. Consequently, we conclude that $\chi((L_f)_0) = \chi((L_g)_0)$.
						
						Moreover, since both $V(f)$ and $V(g)$ have isolated singularities, their local Euler obstructions at the origin coincide. Therefore, the Euler obstructions of $V(f)$ and $V(g)$ are equal.    
			\end{proof}
			
			In \cite{BFS} Bobadilla, Fernandes, and Sampaio proved the following theorem.
			
			\begin{theorem}\label{BFS}
				Let $X\subset \mathbb{C}^{N+1}$ and $Y\subset \mathbb{C}^{M+1}$ be two complex analytic surfaces. If $(X,0)$ and $(Y,0)$ are bi-Lipschitz homeomorphic, then $m_0(X,0)=m_0(Y,0)$.
			\end{theorem}
			
			Then, as a Corollary of this result and from Theorems \ref{teoeuler} and \ref{lt}, we obtain.
			
			\begin{corollary}
				Let $f, g \in \mathcal{O}_3$ be germs with an isolated singularity at the origin such that $\overline{J(f)} = \mathfrak{m}^{\operatorname{ord}(f)-1}$.
				Suppose that $f$ and $g$ are bi-Lipschitz $\mathcal{A}$-equivalent 
				or bi-Lipschitz $\mathcal{K}^*$-equivalent. Then  $m_1(X,0)=m_1(Y,0)$, where $X=f^{-1}(0)$ and $Y=g^{-1}(0)$ .
			\end{corollary}
			\begin{proof}
				Let $f,g \in \mathcal{O}_3$ be germs with isolated singularities at the origin such that 
				$J(f)=\mathfrak{m}^{\mathrm{ord}(f)-1}$, and suppose that $f$ and $g$ are bi-Lipschitz 
				$\mathcal{A}$-equivalent or bi-Lipschitz $\mathcal{K}^*$-equivalent.
				Set $X = f^{-1}(0)$ and $Y = g^{-1}(0)$. Then $X$ and $Y$ are hypersurfaces in $\mathbb{C}^3$, 
				hence complex analytic surfaces.
				By Theorem  \ref{teoeuler}, we have
				$$\mathrm{Eu}_X(0) = \mathrm{Eu}_Y(0).$$
				On the other hand, by Theorem \ref{BFS}, since $(X,0)$ and $(Y,0)$ are bi-Lipschitz homeomorphic, 
				their multiplicities coincide: $$m_0(X,0) = m_0(Y,0).$$
				Now, using the Lê--Teissier formula (Theorem \ref{lt}), for surfaces we have
				$$\mathrm{Eu}_X(0) = (-1)^{2-0-1} m_0(X,0) + (-1)^{2-1-1} m_1(X,0)
				= - m_0(X,0) + m_1(X,0),$$
				and similarly,
				$$\mathrm{Eu}_Y(0) = - m_0(Y,0) + m_1(Y,0).$$
				Since $\mathrm{Eu}_X(0) = \mathrm{Eu}_Y(0)$ and $m_0(X,0)=m_0(Y,0)$, it follows that
				$$- m_0(X,0) + m_1(X,0) = - m_0(Y,0) + m_1(Y,0),$$
				and therefore
				$m_1(X,0) = m_1(Y,0).$
			\end{proof}
			
		\begin{example}
			Let $f(x,y,z) = x^{2} + y^{a} + z^{2} \in \mathcal{O}_{3}$,
		where $a\geq2$, and consider the germ
		$$g(x,y,z) = x^2 + (y + xy)^a + (z + 2y^2 + xy^2)^2 +2(y^2-xz) \in \mathcal{O}_{3}.$$
		
		The quadratic part of $g$ is $(x - z)^2 + 2y^2$, whose critical locus is given by $\{x = z,\ y = 0\}$. A direct computation shows that this curve is contained in $\mathrm{Sing}(V(g))$. Hence $\mathrm{Sing}(V(g))$ has dimension at least one, whereas $\mathrm{Sing}(V(f))$ is isolated.
		
		Since the dimension of the singular locus is preserved under bi-Lipschitz homeomorphisms and
		$$\dim \mathrm{Sing}(V(f)) \neq \dim \mathrm{Sing}(V(g)),$$
		it follows that the germs $(V(f),0)$ and $(V(g),0)$ are not bi-Lipschitz equivalent. In particular, $f$ and $g$ are not bi-Lipschitz $\mathcal{R}$-equivalent.

		We now consider the same germs in the toric setting. Let $X(S)$ be the toric variety associated with the semigroup $S=\{(1,0),(1,1),(1,2)\}$. When $f$ and $g$ are regarded as elements of $\mathcal{O}_{X(S),0}$, the situation changes significantly. Indeed, since $y^2 - xz = 0$ on $X(S)$, the term $2(y^2 - xz)$ vanishes identically. Moreover, the map
		$$\varphi(x,y,z) = \big(x,\; y+xy,\; z+2y^2+xy^2\big)$$
		defines a local bi-Lipschitz homeomorphism of $(X(S),0)$. Therefore,
		$$g = f \circ \varphi \quad \text{in } \mathcal{O}_{X(S),0}.$$
		It follows that $f$ and $g$ are bi-Lipschitz $\mathcal{R}$-equivalent on $(X(S),0)$, and therefore also bi-Lipschitz $\mathcal{A}$-equivalent. 	Consequently, by Theorem \ref{teoeuler},
		$$
		Eu_{V(f)}(0) = Eu_{V(g)}(0).$$
		
		\end{example}
			\begin{example}
				Consider $f(x,y,z) = x^2 + y^2 + z^2 \in \mathcal{O}_3$. We have $\operatorname{ord}(f) = 2$.  The Jacobian ideal is $J(f) = \mathfrak{m}$. Let $\Phi(x,y,z) = (x+y,\, y,\, z).$
				Define $$g(x,y,z) = f(\Phi(x,y,z)) = (x+y)^2 + y^2 + z^2.$$
				Since $\Phi$ is a linear invertible transformation, it is bi-Lipschitz, and thus $f$ and $g$ are bi-Lipschitz equivalent. Therefore, all the assumptions of Theorem \ref{teoeuler} are satisfied, and we conclude that $$Eu_{V(f)}(0)= Eu_{V(g)}(0).$$
			\end{example}

			\begin{center}
				{ \bf Acknowledgments}
			\end{center}
			
			We would like to thank Alexandre Fernandes and Carles Bivià-Ausina for all the helpful discussions they provided throughout the development of this work.



\begin{thebibliography}{99}
				
				\bibitem{ATT} A. S. Araújo, T. M. Dalbelo and T. da Silva, \textit{Newton polyhedra and the integral
				closure of ideals on toric varieties}. ArXiv:2409.07986, (2024).
			
				\bibitem{BirbrairFernandesSampaioVerbitsky2020}
				L. Birbrair, A. Fernandes, J. E. Sampaio, and M. Verbitsky,
				\textit{Multiplicity of singularities is not a bi-Lipschitz invariant},
				Mathematische Annalen, {377}:115–121, (2020).
							
				\bibitem{BFS}
				J. F. Bobadilla, A. Fernandes, and J. E. Sampaio,
				\textit{Multiplicity and degree as bi-Lipschitz invariants for complex sets},
				Journal of Topology, {11}(4):958–966, (2018).
								
				\bibitem{BLS} J.-P. Brasselet, D. T. L\^{e} e J. Seade, \textit{Euler obstruction and indices of vector fields}, Topology, 6, 1193-1208, (2000).
							
				\bibitem{BS} J.-P. Brasselet e M.-H. Schwartz, \textit{Sur les classes de Chern d'un ensemble analytique complexe}, Astérisque \textbf{82-83}, 93-147, (1981).
							
				\bibitem{carlesfukui} C. Bivià-Ausina and T. Fukui. \textit{Invariants for bi-Lipschitz equivalence of ideals.} Q. J.
				Math., 68(3):791–815, (2017).
								
				\bibitem{Cox} D. A. Cox, J. B. Little, and H. K. Schenck, \textit{Toric varieties}. Graduate Studies in Mathematics 124, AMS, (2011). 
								
				\bibitem{FS}
				A. Fernandes and J. E. Sampaio,
				\textit{Multiplicity of analytic hypersurface singularities under bi-Lipschitz homeomorphisms},
				Journal of Topology, {9}:927–933, (2016).
								
				\bibitem{Gaffney1996} T. Gaffney, \textit{Multiplicities and equisingularity of ICIS germs}. Inventiones Mathematicae,
				123(2):209–220, (1996).
								
				\bibitem{hunekeswanson} C. Huneke and I. Swanson, \textit{Integral closure of ideals, rings, and modules,} London Mathematical Society Lecture Note Series 336, Cambridge University Press (2006).
												
				\bibitem{teseNivaldo} N. G. Grulha Jr., \textit{Obstrução de Euler de Aplicações Analíticas}, Tese de Doutorado, ICMC - USP, (2007).
				
								
				\bibitem{Kashiwara} M. Kashiwara, \textit{Systems of microdifferential equations}, Birkhauser (1983).
							
				\bibitem{LT} D. T. L\^{e} and B. Teissier, \textit{Vari\'{e}t\'{e}s polaires locales et classes de Chern des vari\'{e}t\'{e}s singuli\`{e}res}, Ann. of Math. 114, 457-491, (1981).
				
				\bibitem{L-T} M. Lejeune and B. Teissier, \textit{Clôture intégrale des idéaux et équisingularité, with an appendix
				by J. J. Risler.} Centre de Mathématiques, École Polytechnique (1974) and Ann. Fac. Sci.
				Toulouse Math. 17, 781–859,(2008).
								
				\bibitem{MacPherson} R. D. MacPherson, \textit{Chern classes for singular algebraic varieties}, Ann. of Math. 100, 423-432, (1974).
								
				\bibitem{Matsui} Y. Matsui and K. Takeuchi. \textit{Milnor fibers over singular toric varieties and nearby cycle sheaves}. Tohoku Math. J., 63(1):113–136, (2011).
						
				\bibitem{Matsumura} H. Matsumura, \textit{Commutative ring theory}, Cambridge studies in advanced mathematics no. 8, Cambridge University Press, (1986).
				
				\bibitem{NM} J. J. Nuño-Ballesteros and R. Mendes, \textit{Topological Classification and Finite Determinacy of Knotted Maps}, Michigan Math. J. 69(4): 831-848 (2020).
								
				\bibitem{oka} M. Oka. \textit{Non-degenerate complete intersection singularity}. Hermann, Paris, (1997).
				
				\bibitem{Saia} M. J. Saia, \textit{The integral closure of ideals and the Newton filtration}, J. Algebraic Geom. 5,  1--11 (1996).
				
			\end{thebibliography}
		\end{document}